\documentclass{article}
\usepackage[utf8]{inputenc}
\usepackage{amsmath,amssymb,amsthm}
\usepackage{authblk}
\usepackage{hyperref}
\usepackage{url}
\usepackage{todonotes}

\usepackage{fourier}

\newcommand{\ev}{\mathbb{E}}
\newcommand{\prob}{\mathbb{P}}

\newtheorem{theorem}{Theorem}[section]
\newtheorem{lemma}[theorem]{Lemma}

\title{On the number of finite additive 2-bases}
\author{Stefan Weltge}
\author{Konrad Zyhalko}
\affil{\small \it Technical University of Munich, Germany}
\date{}

\begin{document}
\maketitle

\begin{abstract}
    The number of finite additive 2-bases is known to grow exponentially. While this fact has been established by Marzuola and Miller~(2010) using complex analytic techniques embedded in the study of numerical sets, we provide a direct, short proof using elementary probabilistic arguments.
\end{abstract}

\section{Introduction}

For a non-negative integer $n$, let $[n]_0 := \{0,1,\dots,n\}$.
An \emph{additive 2-base} for $n$ is a subset $X \subseteq [n]_0$ such that $[n]_0 \subseteq X + X$, where $X + X = \{x + y : x,y \in X\}$.
The study of additive 2-bases is a classical topic in additive number theory, initiated by Rohrbach~\cite{Rohrbach} whose work focused on the smallest cardinality of an additive 2-base for $n$.
The minimal size of an additive 2-base is known to grow asymptotically as $\sqrt{n}$, and more precise bounds have been established over the years by various authors (see \cite{HammererHofmeister}, \cite{Mrose}, \cite{Moser}, \cite{GunturkNathanson}, \cite{Yu}).

In this short note, we are interested in the number of additive 2-bases for $n$.
To this end, let $\Gamma(n)$ be the family of all additive 2-bases for $n$.
The values $|\Gamma(n)|$ for $n = 0,\dots,65$ were calculated by Martin Fuller and are indexed in OEIS~\cite[A066062]{oeis}.
The principal theoretical work yielding a strong lower bound on $|\Gamma(n)|$ is due to Marzuola and Miller~\cite{MarzuolaMiller}, who primarily study ``numerical sets with no small atoms'' and show that this notion is intimately related to additive $2$-bases.
As a consequence, they find that $|\Gamma(n)|$ grows exponentially in $n$ and obtain a strong estimate for the asymptotic constant.
However, due to their focus on numerical sets, their approach involves significant complexity.

In this short note, we provide a short and simple proof for the fact that the number of additive 2-bases for $n$ grows exponentially in $n$.

\begin{theorem}
    \label{thmMain}
    There exists some $\alpha > 0$ such that $|\Gamma(n)| \ge \alpha \cdot 2^{n+1}$ holds for all integers $n \ge 0$.
\end{theorem}

In what follows, we focus on simplicity rather than optimizing constants, and hence most bounds we obtain are not tight.

\section{Exponential lower bound}

In all subsequent statements, $n$ is assumed to be a non-negative integer.

\begin{lemma}
    \label{lemNotGenerated}
    If $X$ is a uniformly random subset of $[n]_0$ and $k \in [n + 1]_0$, then
    \[
        \prob[k \notin X + X] \le \left( \frac{3}{4} \right)^{\frac{k}{2} - 1}.
    \]
\end{lemma}
\begin{proof}
    For $k=0,1$, the claimed bound is trivial. Letting $E_i$ denote the event that $\{i,k-i\} \not \subseteq X$, we have for $k \ge 2$ that
    \begin{align*}
        \prob[k \notin X + X]
        = \prob[E_0\wedge E_1 \wedge \dots \wedge E_{\lfloor k/2 \rfloor}]
        & \le \prob[E_1 \wedge E_2 \wedge \dots \wedge E_{\lceil k/2 \rceil - 1}] \\
        & = \prod_{i=1}^{\lceil k/2 \rceil - 1} \prob[E_i],
    \end{align*}
    where the second equality follows from the mutual independence of the events.
    The claim follows from the observation that $\prob[E_i] = \frac{3}{4}$ for each of the above events, whenever $1 \le i \le \lceil k/2 \rceil - 1$ and $k \le n + 1$.
\end{proof}

\begin{lemma}
    \label{lemRandom}
    If $X$ is a uniformly random subset of $[n]_0$, then the expected number of integers in $[n]_0 \setminus (X + X)$ is at most $10$.
\end{lemma}
\begin{proof}
    For $k \in [n]_0$, let $Y_k$ be the indicator random variable, that is $1$ if $k \notin X + X$, and $0$ otherwise.
    Note that $|[n]_0 \setminus (X + X)| = \sum_{k=0}^n Y_k$ and hence using Lemma~\ref{lemNotGenerated} we obtain
    \begin{multline*}
        \ev[|[n]_0 \setminus (X + X)|]
            = \sum_{k=0}^n \ev[Y_k]
            = \sum_{k=0}^n \prob[k \notin X + X]
            \le \sum_{k=0}^n \left( \frac{3}{4} \right)^{\frac{k}{2} - 1} \\
            = \frac{4}{3} \cdot \frac{1 - \left( \sqrt{3/4} \right)^{n+1}}{1 - \sqrt{3/4}}
            \le \frac{4}{3} \cdot \frac{1}{1 - \sqrt{3/4}}
            \le 10.
    \end{multline*}
\end{proof}

\begin{lemma}
    \label{lemFirstLowerBound}
    We have $|\Gamma(n)| \ge 2^n / (n+1)^{20}$.
\end{lemma}
\begin{proof}
    Let $\Gamma'(n)$ be the family of subsets $X \subseteq [n]_0$ such that $X$ fails to generate at most $20$ integers in $[n]_0$, i.e., $|[n]_0 \setminus (X + X)| \le 20$.
    To estimate the size of $\Gamma'(n)$, let $X$ be a uniformly random subset of $[n]_0$ and $Y = |[n]_0 \setminus (X + X)|$.
    By Lemma~\ref{lemRandom} and Markov's inequality, we have
    \[
        \prob[Y > 20] \le \frac{\ev[Y]}{20} \le \frac{1}{2}.
    \]
    Therefore,
    \begin{equation}
        \label{eqBoundGammaPrime1}
        |\Gamma'(n)| = (1 - \prob[Y > 20]) \cdot 2^{n+1} \ge 2^n.
    \end{equation}
    Now, to every set $X \in \Gamma'(n)$ add the (at most $20$) elements from $[n]_0$ that are not in $X + X$.
    The resulting set is in $\Gamma(n)$.
    Moreover, notice that a single set in $\Gamma(n)$ can be obtained in this way from at most $(n+1)^{20}$ sets in $\Gamma'(n)$, and hence we have
    \begin{equation}
        \label{eqBoundGammaPrime2}
        |\Gamma'(n)| \le (n+1)^{20} \cdot |\Gamma(n)|.
    \end{equation}
    The claim follows by combining inequalities~\eqref{eqBoundGammaPrime1} and~\eqref{eqBoundGammaPrime2}.
\end{proof}

\begin{lemma}
    \label{lemRecurrence}
    For every integer $n \ge 0$, we have
    \[
        |\Gamma(n+1)| = |\Gamma(n)| + |\{X\in \Gamma(n) : n + 1 \in X+X\}|.
    \]
\end{lemma}
\begin{proof}
    The claim follows from the observations that
    \[
        \{ X \in \Gamma(n+1) : n + 1 \in X \}
        = \{ X \cup \{n + 1\} : X \in \Gamma(n) \}
    \]
    and
    \[
        \{ X \in \Gamma(n+1) : n + 1 \notin X \}
        = \{ X \in \Gamma(n) : n + 1 \in X + X\}. \qedhere
    \]
\end{proof}

\begin{proof}[Proof of Theorem~\ref{thmMain}]
    Consider the ratio
    \[
        \delta(n) := \frac{|\{X \in \Gamma(n) : n + 1 \in X + X\}|}{|\Gamma(n)|}.
    \]
    By Lemma~\ref{lemRecurrence}, we have
    \begin{equation}
        \label{eqRecurrenceDelta}
        |\Gamma(n+1)| = (1 + \delta(n)) \cdot |\Gamma(n)|.
    \end{equation}
    By Lemma~\ref{lemNotGenerated} and Lemma~\ref{lemFirstLowerBound}, we have
    \begin{align*}
        1 - \delta(n)
        = \frac{|\{X \in \Gamma(n) : n + 1 \notin X + X\}|}{|\Gamma(n)|}
        & \le \frac{|\{X \subseteq [n]_0 : n + 1 \notin X + X\}|}{|\Gamma(n)|} \\
        & \le 2 \cdot \left( \frac{3}{4} \right)^{\frac{n-1}{2}} \cdot (n+1)^{20},
    \end{align*}
    and hence
    \begin{equation}
        \label{eqDeltaBound}
        1 + \delta(n) \ge 2 - \underbrace{\tfrac{4}{\sqrt{3}} \left( \tfrac{\sqrt{3}}{2} \right)^n \cdot (n+1)^{20}}_{=: t(n)}.
    \end{equation}
    Note that
    \[
        \frac{t(n+1)}{t(n)} = \frac{\sqrt{3}}{2} \cdot \left( 1 + \frac{1}{n+1} \right)^{20},
    \]
    and hence there exists some integer $n_0$ such that
    \begin{equation}
        \label{eqLe1}
        t(n) \le \frac{1}{10}
    \end{equation}
    and
    \begin{equation}
        \label{eqGeometricDecay}
        t(n + 1) \le \frac{9}{10} t(n)
    \end{equation}
    holds for all $n \ge n_0$.
    We conclude
    \begin{align*}
        \frac{|\Gamma(n)|}{2^n}
        = \frac{|\Gamma(n_0)|}{2^{n}} \cdot \prod_{k=n_0}^{n - 1} (1 + \delta(k))
        & \ge \frac{|\Gamma(n_0)|}{2^{n_0}} \cdot \prod_{k=n_0}^{n - 1} \left( 1 - \frac{t(k)}{2} \right) \\
        & \ge \frac{|\Gamma(n_0)|}{2^{n_0}} \cdot \left( 1 - \frac{1}{2} \sum_{k=n_0}^{n - 1} t(k) \right) \\
        & \ge \frac{|\Gamma(n_0)|}{2^{n_0}} \cdot \left( 1 - \frac{t(n_0)}{2} \sum_{k=0}^{n - n_0 - 1} \left( \frac{9}{10}\right)^k \right) \\
        & \ge \frac{|\Gamma(n_0)|}{2^{n_0}} \cdot \left( 1 - 5 \cdot t(n_0) \right) \\
        & \ge \frac{|\Gamma(n_0)|}{2^{n_0 + 1}},
    \end{align*}    
    where the first equality follows from \eqref{eqRecurrenceDelta}, the first inequality from~\eqref{eqDeltaBound} and~\eqref{eqLe1}, the third inequality from~\eqref{eqGeometricDecay}, and the last inequality from~\eqref{eqLe1}.
\end{proof}

\bibliographystyle{plain}          
\bibliography{references}   

\end{document}